\documentclass[12 pt]{article}
\usepackage[a4paper, total={6.5in, 9.5in}]{geometry}
\usepackage[T1]{fontenc}
\usepackage[utf8]{inputenc}
\usepackage[english]{babel}

\usepackage{graphicx}

\usepackage{amsmath}
\usepackage{amsfonts}
\usepackage{amssymb}
\usepackage{amsthm}
\usepackage[alphabetic]{amsrefs}

\usepackage{mathtools}

\usepackage[colorlinks=true,linkcolor=blue,urlcolor=blue,citecolor=blue]{hyperref}
\newtheorem{theorem}{Theorem}[section]

\newtheorem{lemma}[theorem]{Lemma}
\newtheorem{proposition}[theorem]{Proposition}

\newtheorem{corollary}[theorem]{Corollary}
\numberwithin{equation}{section}

\title{Local density of Caputo-stationary functions \\ of any order\thanks{Supported by
the Australian Research Council Discovery Project 170104880 NEW ``Nonlocal
Equations at Work''. The authors are members of INdAM/GNAMPA.}}
\author{Alessandro Carbotti\thanks{Dipartimento di Matematica
e Fisica, Universit\`a del Salento,
Via Per Arnesano, 73100 Lecce, Italy. {\tt alessandro.carbotti@unisalento.it}}, Serena Dipierro\thanks{Department
of Mathematics and Statistics,
University of Western Australia,
35 Stirling Highway,
Crawley WA 6009, Australia. {\tt serena.dipierro@uwa.edu.au} },
and
Enrico Valdinoci\thanks{Department of Mathematics and Statistics,
University of Western Australia,
35 Stirling Highway,
Crawley WA 6009, Australia, 
and Istituto di Matematica Applicata e Tecnologie Informatiche,
Consiglio Nazionale delle Ricerche,
Via Ferrata 1, 27100 Pavia, Italy,
and Dipartimento di Matematica, Universit\`a degli studi di Milano,
Via Saldini 50, 20133 Milan, Italy. {\tt enrico@mat.uniroma3.it} }}

\begin{document}
\maketitle

\begin{abstract}
We show that any given function can be approximated with arbitrary precision
by solutions of linear, time-fractional equations of any prescribed order.

This extends a recent result by Claudia Bucur, which was obtained
for time-fractional derivatives of order less than one, to the case of any fractional
order of differentiation.

In addition, our result applies also to the $\psi$-Caputo-stationary case,
and it will provide one of the building blocks of a forthcoming paper in 
which we will establish general approximation results by operators of any order
involving anisotropic
superpositions of classical, space-fractional and time-fractional diffusions.
\end{abstract}

\section{Introduction}

The goal of this paper is to show that {\em every function
can be approximated arbitrarily well by functions
that satisfy a homogeneous equation driven by the Caputo derivative
of any positive order}.

When the order of the Caputo derivative
is less than one, this type of statement has been recently proved
in~\cite{MR3716924}. In this sense, our result here
is the extension of the main theorem in~\cite{MR3716924} to
higher order Caputo fractional derivatives.\medskip

More precisely,
we denote by~$\mathbb{N}$ the set of natural numbers
(starting from~$1$) and~$\mathbb{N}_0:=\mathbb{N}\cup\{0\}$
(in this way, $\mathbb{N}=\mathbb{N}_0\setminus\{0\}$).
Following~\cite{MR2379269},
given
\begin{equation}\label{SETTI}
{\mbox{$a\in\mathbb{R}\cup\left\{-\infty\right\}$, $k\in\mathbb{N}$
and~$\alpha\in(k-1,k)$, }}\end{equation}
one defines the Caputo derivative
of initial point $a$ and order~$\alpha$ by
\begin{equation}
\label{defcap}
D_a^{\alpha}u(t):=\frac{1}{\Gamma(k-\alpha)}
\int_a^t \frac{u^{(k)}\left(\tau\right)}{(t-\tau)^{\alpha-k+1}} d\tau,
\end{equation}
where~$\Gamma$ is the Euler's Gamma-Function.

In addition, if~$I\subseteq\mathbb{R}$ is an interval, 
we define the space 
\begin{equation*}
AC^{k-1}(I):=\left\{f\in C^{k-1}(I){\mbox{ s.t. }}f,f',\ldots, f^{(k-1)}\in AC(I)\right\},
\end{equation*}
where $C^{k-1}(I)$ denotes the space of $(k-1)$-times
continuously differentiable functions on $I$,
and $AC(I)$ denotes the space of absolutely continuous functions on $I$.

Given~$t>a$, 
$k\in\mathbb{N}$, $\beta>0$, and~$f:[a,+\infty)\rightarrow\mathbb{R}$,
we also define the function
\begin{equation}\label{CkAM1} (a,t)\ni\tau\,\mapsto\,\Theta_{k,\beta,f,t}(\tau):=
f^{(k)}(\tau)(t-\tau)^{k-\beta-1}\end{equation}
and we set
\begin{equation}\label{CkAM2}\begin{split}
C^{k,\beta}_a\,&:=\Big\{f:\overline{(a,+\infty)}\rightarrow\mathbb{R}
\;\text{ s.t. }\; f\in AC^{k-1}\big(\overline{(a,t)}\big)\\&\qquad\quad \text{and}\quad
\Theta_{k,\beta,f,t}\in L^1\big((a,t)\big),\; {\mbox{ for all }} t>a\Big\}.\end{split}
\end{equation}
We observe that the Caputo derivative in~\eqref{defcap}
is well defined for all~$u$ belonging to~$C^{k,\alpha}_a$.
In this setting, we have the following density result:

\begin{theorem}
\label{dens}
Let $h\in\mathbb{N}_0$, $k\in\mathbb{N}$, and~$\alpha\in(k-1,k)$. 

Then, for every $f\in C^h([0,1])$ and $\epsilon>0$, there exist $a<0$
and $u\in C^{k,\alpha}_a$ such that 
\begin{eqnarray}
\label{CLAIM1}&&D_a^\alpha u(t)=0\quad\text{in}\quad[0,+\infty)\\
\label{CLAIM2}{\mbox{and }}&&\left\|u-f\right\|_{C^h([0,1])}<\epsilon.\end{eqnarray}
\end{theorem}

Theorem~\ref{dens} lies in the research line
of approximation results with solutions of nonlocal equations.
The first result in this direction was obtained in~\cite{MR3626547},
where it was established that any given function can be locally
approximated with arbitrary precision by functions with vanishing
fractional Laplacian. This result has been also extended
in~\cite{DSV1}
to take into account also evolution equations, and in general equations
which contain different types of diffusion in different coordinate variables.

When~$\alpha\in(0,1)$, Theorem~\ref{dens} has been recently
obtained in~\cite{MR3716924}. Furthermore,
Theorem~\ref{dens} will constitute one of the main building blocks
for the forthcoming paper~\cite{FUTURO},
in which we will establish a very general density result
for solutions of operators taking into account both classical
and fractional derivatives of any order and of both time-fractional
and space-fractional types.
\medskip

We also remark that Caputo derivatives possess a number
of concrete applications in describing processes
with memory, see e.g.~\cite{MR1918790} and the references therein,
hence we think that it is quite interesting that the set of solutions
of linear Caputo-type equations is shown by Theorem~\ref{dens}
to be so abundant to shadow the profile of any prescribed function, also
independently on any geometric constraint.\medskip

The proof of Theorem~\ref{dens} will rely on an appropriate ``derivative spanning''
method introduced in~\cite{MR3626547}. Roughly speaking,
the nonlocal effect produced by the operator causes a fractional
type boundary behavior which is persistent for all the derivatives
of the solutions. Then, this phenomenon in turn implies that the derivative jet
at a given point is essentially arbitrary, and the desired
result follows by rescaling.

To make such argument work, one needs
to construct a suitable solution with a very precise control on every boundary derivatives.
In our setting, this goal will be achieved by a careful analysis
of the linear equation, in terms of explicit representation formulas
and asymptotic analysis.\medskip

It is also interesting to point out that a simple variant of Theorem~\ref{dens}
comprises the case of more general
nonlocal operators of time-fractional type. To describe
this generalized setting, we consider the functional space
\begin{equation}\label{DEPSI}
\Psi_k:=\Big\{
\psi\in C^k({\mathbb{R}}) {\mbox{ s.t.
$\psi'(x)> 0$ for any~$x\in{\mathbb{R}}$}}
\Big\}.
\end{equation}
Given~$\psi\in\Psi_k$,
following~\cites{MR3554830, MR3072517},
one can introduce a time-fractional derivative
with respect to $\psi$ and initial point~$a\in{\mathbb{R}}\cup\{-\infty\}$, 
defined by
\begin{equation}\label{caputotype}
D_a^{\alpha,\psi} u\left(t\right):=\frac{1}{\Gamma\left(k-\alpha\right)}\int_a^t {\frac{u^{\left(
k\right)}_\psi\left(\tau\right)}{\left(\psi\left(t\right)-\psi\left(\tau\right)\right)^{\alpha-k+1}} \psi'\left(\tau\right)\,d\tau}
\end{equation}
where 
\begin{equation*}
u^{(k)}_\psi(\tau)
:=\left(\frac{1}{\psi'(\tau)}\frac{d}{d\tau}\right)^{k}u(\tau)=
\underbrace{
\left(\frac{1}{\psi'(\tau)}\frac{d}{d\tau}\right)\cdots \left(\frac{1}{\psi'(\tau)}\frac{d}{d\tau}\right)
}_{{\mbox{$k$ times}}}
u(\tau).
\end{equation*}
The setting in~\eqref{caputotype} 
comprises, as a particular case,
the Caputo
fractional derivative
given by \eqref{defcap}
(notice indeed that~\eqref{caputotype} reduces to~\eqref{defcap}
when~$\psi(t):=t$).\medskip

In this context, we obtain 
from Theorem~\ref{dens} that $\psi$-Caputo-stationary
functions are locally dense, in the sense made precise
by the following result:

\begin{corollary}
\label{densepsi} Let~$h\in\mathbb{N}_0$, $k\in\mathbb{N}$, and
$\alpha\in(k-1,k)$. Let also~$\psi\in\Psi_k\cap C^h\left([0,1]\right)$.

Assume that
\begin{equation}\label{CPSI}
\lim_{t\to-\infty}\psi(t)=-\infty.\end{equation}
Then,
for any~$\epsilon>0$ and any~$f\in C^h\left([0,1]\right)$
there exist an initial point~$ a \in(-\infty,0)$
and a function~$u\in C^{k,\alpha}_a$ which satisfies
\begin{eqnarray*}&&
D^{\alpha,\psi}_a u(t) = 0 \quad {\mbox{ for any }} t\in[0,+\infty)\\
{\mbox{and }}&& \|u - f\|_{C^h\left([0,1]\right)}<\epsilon.
\end{eqnarray*}\end{corollary}

In the forthcoming Section~\ref{9ikHNA:SP}, we will prove
Theorem~\ref{dens}. {F}rom this, we will derive the proof
of Corollary~\ref{densepsi} in Section~\ref{9ikHNA:SP2}.

\section{Proof of Theorem~\ref{dens}}\label{9ikHNA:SP}

\subsection{Existence, uniqueness and regularity of solutions of time-fractional equations}

The proof of Theorem~\ref{dens} relies on a series of auxiliary results
(here, if not specified, we always implicitly
suppose that the setting in~\eqref{SETTI}
is assumed).
We start with an equivalent formulation of time-fractional
equations which highlights the role played by the memory effect:
roughly speaking, solving a homogeneous
time-fractional equation with some
given initial data at~$t=a$ is equivalent to solving
a non-homogeneous
time-fractional equation with initial time~$t=b>a$,
and the non-homogeneous source in the equation takes into account
the memory effect of the period of time~$t\in[a,b]$.
The precise result that we need is the following:

\begin{lemma}
\label{equiv}
Let $b>a$ and~$\varphi\in C^k([a,b])$.
Then, $u\in C^{k,\alpha}_a$ satisfies the equation
\begin{equation*}
\begin{cases}
D^{\alpha}_a u(t)=0&\quad\text{in}\quad(b,+\infty) ,\\
u(t)=\varphi(t)&\quad\text{in}\quad(-\infty,b],
\end{cases}
\end{equation*}
if and only if it satisfies
\begin{equation*}
\begin{cases}
D^{\alpha}_b u(t)=g(t)&\quad\text{in}\quad(b,+\infty) ,\\
u(t)=\varphi(t)&\quad\text{in}\quad(-\infty,b],
\end{cases}
\end{equation*}
where
\begin{equation}
\label{gi}
g(t):=-\frac1{\Gamma(k-\alpha)}\int_a^b \frac{\varphi^{(k)}(\tau)}{(t-\tau)^{\alpha-k+1}} d\tau.
\end{equation}
\begin{proof}
First of all, we observe that the function $g$ in~\eqref{gi}
is well defined, since
\begin{equation*}
\Gamma(k-\alpha)\,|g(t)|\leq\int_a^b \frac{|\varphi^{(k)}(\tau)|}{(t-\tau)^{\alpha-k+1}} d\tau\leq\sup_{\tau\in[a,b]}|\varphi^{(k)}(\tau)|\frac{(t-a)^{k-\alpha}-(t-b)^{k-\alpha}}{k-\alpha},
\end{equation*}
which is finite for any $t\in[b,+\infty)$. 

Using the definition in~\eqref{defcap}, we have that
\begin{align*}
D^\alpha_a u(t)&=\frac1{\Gamma(k-\alpha)}
\int_b^t u^{(k)}(\tau)(t-\tau)^{k-\alpha-1} d\tau+\frac1{\Gamma(k-\alpha)}\int_a^b u^{(k)}(\tau)(t-\tau)^{k-\alpha-1} d\tau \\
&=\frac1{\Gamma(k-\alpha)}\int_b^t u^{(k)}(\tau)(t-\tau)^{k-\alpha-1} d\tau+\frac1{\Gamma(k-\alpha)}\int_a^b \varphi^{(k)}(\tau)(t-\tau)^{k-\alpha-1} d\tau\\
&=D^\alpha_b u(t)-g(t).
\end{align*}
{F}rom this, the desired result plainly follows.
\end{proof}
\end{lemma}

Next lemma gives a representation formula
for the solutions of time-fractional equations of any order
(when the order is below $1$, such a result
is related to Volterra-type integral equations,
and we provide the details for any order for the facility of the reader).
The main representation formula will be given in the forthcoming Lemma~\ref{VOLT}.
To this end, we present some ancillary observation on the derivatives
of integral identities:

\begin{lemma}\label{ANCD}
Let~$g\in C_a^{k,k-\alpha}$, and let, for any~$t\ge a$,
$$ v(t):=\int_a^t g(\tau)(t-\tau)^{\alpha-1} d\tau.$$
Then, $v\in AC^{k-1}([a,+\infty))$, and, for any~$t>a$,
\begin{equation}\label{Cva} v^{(k)}(t)=
\int_a^t g^{(k)}(\tau)(t-\tau)^{\alpha-1}d\tau+
\sum_{i=0}^{k-1} 
\frac{(\alpha+i)\dots(\alpha+i-k+1)\,g^{(i)}(a)}{\alpha(\alpha+1)\dots(\alpha+i)}
\;(t-a)^{\alpha+i-k}.\end{equation}

\begin{proof} Using recursively
the Fundamental Theorem of Calculus for absolutely continuous functions
(see e.g. Theorem~7.6 in~\cite{MR3379909}), we can write
\begin{equation}\label{76}
g(t):=\frac{1}{(k-1)!}\int_a^t g^{(k)}(\sigma)(t-\sigma)^{k-1} \,d\sigma+
\sum_{i=0}^{k-1} \frac{g^{(i)}(a)}{i!}(t-a)^i.
\end{equation}
As a result, we can write
\begin{equation}\label{FUb}\begin{split}&
v(t)=\frac1{(k-1)!}\int_a^t 
\left[\int_a^\tau g^{(k)}(\sigma)(\tau-\sigma)^{k-1} \,d\sigma\right]
(t-\tau)^{\alpha-1} \,d\tau\\&\qquad\qquad\qquad+
\sum_{i=0}^{k-1}\frac{g^{(i)}(a)}{i!} \int_a^t (\tau-a)^i(t-\tau)^{\alpha-1}\,d\tau.
\end{split}\end{equation}
We also remark that, for every~$\sigma\in[a,t)$,
\begin{equation}\label{OJLSNNDND}
\begin{split}
& \int_\sigma^t (\tau-\sigma)^{k-1} (t-\tau)^{\alpha-1} \,d\tau
=-\frac1\alpha\int_\sigma^t (\tau-\sigma)^{k-1} \frac{d}{d\tau}(t-\tau)^{\alpha} \,d\tau
\\&\qquad=
\frac{k-1}\alpha\int_\sigma^t (\tau-\sigma)^{k-2} (t-\tau)^{\alpha} \,d\tau=
-\frac{k-1}{\alpha(\alpha+1)}
\int_\sigma^t (\tau-\sigma)^{k-2} \frac{d}{d\tau}(t-\tau)^{\alpha+1} \,d\tau\\
&\qquad=
\frac{(k-1)(k-2)}{\alpha(\alpha+1)}
\int_\sigma^t (\tau-\sigma)^{k-3} (t-\tau)^{\alpha+1} \,d\tau=\dots\\&\qquad=
\frac{(k-1)(k-2)\dots(k-\ell)}{\alpha(\alpha+1)\dots(\alpha+\ell-1)}
\int_\sigma^t (\tau-\sigma)^{k-\ell-1} (t-\tau)^{\alpha+\ell-1} \,d\tau
\\&\qquad=
\frac{(k-1)!}{\alpha(\alpha+1)\dots(\alpha+k-2)}
\int_\sigma^t (t-\tau)^{\alpha+k-2} \,d\tau\\&\qquad=
\frac{(k-1)!}{\alpha(\alpha+1)\dots(\alpha+k-1)}\,(t-\sigma)^{\alpha+k-1}.
\end{split}\end{equation}
In addition, since~$\Theta_{k,k-\alpha,g,t}\in L^1((a,t))$ for all~$t>a$
in view of~\eqref{CkAM1} and~\eqref{CkAM2}, we know that the map
$$  (a,t)\ni\sigma\,\mapsto\,g^{(k)}(\sigma)\,(t-\sigma)^{\alpha+k-1}$$
belongs to~$L^1((a,t))$. Hence, we can exploit~\eqref{OJLSNNDND} and Fubini's Theorem
to see that
\begin{eqnarray*}&&\frac{(k-1)!}{\alpha(\alpha+1)\dots(\alpha+k-1)}\,
\int_a^t g^{(k)}(\sigma)\,(t-\sigma)^{\alpha+k-1}\,d\sigma\\
&=&\int_a^t g^{(k)}(\sigma)\,
\left[
\int_\sigma^t (\tau-\sigma)^{k-1} (t-\tau)^{\alpha-1} \,d\tau\right]
\,d\sigma\\
&=&\int_a^t 
\left[
\int_a^\tau g^{(k)}(\sigma)\,(\tau-\sigma)^{k-1} \,d\sigma\right]\,(t-\tau)^{\alpha-1}
\,d\tau.
\end{eqnarray*}
Plugging this information into~\eqref{FUb}, and using also~\eqref{OJLSNNDND}
once again, we conclude that, for every~$t>a$,
\begin{eqnarray*}
v(t)&=&
\frac{1}{\alpha(\alpha+1)\dots(\alpha+k-1)}\,
\int_a^t g^{(k)}(\sigma)\,(t-\sigma)^{\alpha+k-1}\,d\sigma
\\&&\qquad +
\sum_{i=0}^{k-1} \frac{g^{(i)}(a)}{\alpha(\alpha+1)\dots(\alpha+i)}\, (t-a)^{\alpha+i}
.
\end{eqnarray*}
We can now take derivatives and find that, for each~$j\in\{0,\dots,k-1\}$,
\begin{equation}\label{AKSMSxS}\begin{split}
v^{(j)}(t)\,&=
\frac{1}{\alpha(\alpha+1)\dots(\alpha+k-1-j)}\,
\int_a^t g^{(k)}(\sigma)\,(t-\sigma)^{\alpha+k-1-j}\,d\sigma\\&\qquad
+
\sum_{i=0}^{k-1} 
\frac{(\alpha+i)\dots(\alpha+i-j+1)\,g^{(i)}(a)}{\alpha(\alpha+1)\dots(\alpha+i)}
\,(t-a)^{\alpha+i-j}.
\end{split}\end{equation}
In particular, when~$j:=k-1$,
\begin{eqnarray*}
v^{(k-1)}(t)&=&
\frac{1}{\alpha}\,
\int_a^t g^{(k)}(\sigma)\,(t-\sigma)^{\alpha}\,d\sigma
+
\sum_{i=0}^{k-1} 
\frac{(\alpha+i)\dots(\alpha+i-k+2)\,g^{(i)}(a)}{\alpha(\alpha+1)\dots(\alpha+i)}\, (t-a)^{\alpha+i-k+1}.
\end{eqnarray*}
Taking one more derivative, we obtain~\eqref{Cva}, as desired.
Also, from~\eqref{Cva} and~\eqref{AKSMSxS}, we see that~$v\in AC^{k-1}
([a,+\infty))$.
\end{proof}
\end{lemma}

Now we state a representation result for linear fractional equations
of any order.

\begin{lemma}\label{VOLT}
Let $g\in C_b^{k,k-\alpha}$. The problem
\begin{equation}
\label{amam}
\begin{cases}
D_b^\alpha u(t)=g(t)&\quad{\mbox{ in }}\quad(b,+\infty) ,\\
u^{(h)}(b)=0&\quad{\mbox{ for any }}\quad h=0,\ldots,k-1,
\end{cases}
\end{equation}
admits a unique solution $u\in C_b^{k,\alpha}$. Moreover, for any $t>b$,
\begin{equation}\label{a.asj}
u(t)=\frac{1}{\Gamma(\alpha)}
\int_b^t g(\tau)(t-\tau)^{\alpha-1} d\tau.
\end{equation}
\begin{proof}
Let us start by proving the uniqueness claim.
For this, let $u_1$ and~$u_2\in C_b^{k,\alpha}$ be two different solutions of \eqref{amam},
and let $u:=u_1-u_2$. Then $u\in C_b^{k,\alpha}$
and
\begin{equation}\label{AJmA9kA}
\begin{cases}
D_b^\alpha u(t)=0&\quad\text{in}\quad(b,+\infty) ,\\
u^{(h)}(b)=0&\quad\text{for any}\quad h=0,\ldots,k-1.
\end{cases}
\end{equation}
Hence, recalling~\eqref{defcap}, for every~$s>t$ we have that
\begin{equation}
\label{nain}0={\Gamma(k-\alpha)}\,
\int_b^s D_b^{\alpha}u(t)\,(s-t)^{\alpha-k}\,dt=\int_b^s\left[
\int_b^t \frac{ u^{(k)}(\tau) }{ (t-\tau)^{\alpha-k+1}} d\tau \right](s-t)^{\alpha-k}
\,dt.\end{equation}
We recall the Euler's Beta-Function, defined, for~$x$ and~$y>0$, as
$$ {\mathrm{B}} (x,y):=\int _{0}^{1} \vartheta^{x-1}(1-\vartheta)^{y-1}\,d\vartheta,$$
and the fact that
\begin{equation}\label{bega}{\mathrm{B}} (x,y)={\mathrm{B}} (y,x)
=\frac{\Gamma (x)\,\Gamma (y)}{\Gamma (x+y)},\end{equation}
see e.g.~\cite{MR0167642}.

We employ the change of variable~$\vartheta:=\frac{t-\tau}{s-\tau}$
and we observe that, for all~$x$, $y>0$,
\begin{equation}\label{BETA} \int_\tau^s (t-\tau)^{x-1}(s-t)^{y-1}\,dt=
(s-\tau)^{x+y-1}\,{\mathrm{B}} (x,y),\end{equation}
and therefore
$$ \int_\tau^s (t-\tau)^{k-\alpha-1}(s-t)^{\alpha-k}dt=
{\mathrm{B}} (k-\alpha,\alpha-k+1)=
\Gamma(k-\alpha)\Gamma(\alpha-k+1).$$

Then, 
\begin{eqnarray*}&&
\Gamma(k-\alpha)\Gamma(\alpha-k+1)\,\|u^{(k)}\|_{L^1((b,s))}=
\Gamma(k-\alpha)\Gamma(\alpha-k+1)\,\int_b^s|u^{(k)}(\tau)|\,d\tau\\&&\qquad=
\int_b^s\left[\int_\tau^s (t-\tau)^{k-\alpha-1}(s-t)^{\alpha-k}dt\right]\,|u^{(k)}(\tau)|\,d\tau.
\end{eqnarray*}
As a result, by Fubini-Tonelli's Theorem,
\begin{eqnarray*}&&
\Gamma(k-\alpha)\Gamma(\alpha-k+1)\,\Big( u^{(k-1)}(s)-u^{(k-1)}(b)\Big)=
\Gamma(k-\alpha)\Gamma(\alpha-k+1)\,\int_b^su^{(k)}(\tau)\,d\tau\\&&\qquad=
\int_b^s\left[\int_\tau^s (t-\tau)^{k-\alpha-1}(s-t)^{\alpha-k}dt\right]\,u^{(k)}(\tau)\,d\tau
\\&&\qquad=\int_b^s\left[\int_b^t (t-\tau)^{k-\alpha-1}
\,u^{(k)}(\tau)d\tau\right](s-t)^{\alpha-k}
\,dt.
\end{eqnarray*}
The latter term vanishes, in the light of~\eqref{nain}, and therefore we conclude that
$$ u^{(k-1)}(s)-u^{(k-1)}(b)=0.$$
Recalling the initial condition in~\eqref{AJmA9kA}, we thereby obtain that $u^{(k-1)}(s)=0$.
Since this is valid for all~$s>b$, we have that~$u^{(k-1)}$ vanishes identically in~$[b,+\infty)$.

This in turn implies that~$u^{(k-2)}$ is constant in~$[b,+\infty)$. Recalling
the initial condition in~\eqref{AJmA9kA}, we thus deduce that~$u^{(k-2)}$ vanishes identically in~$[b,+\infty)$.

Iterating this argument, we find that~$u$ vanishes identically in~$[b,+\infty)$,
and therefore~$u_1$ coincides with~$u_2$ and the uniqueness claim
in Lemma~\ref{VOLT} is established.

To complete the proof of Lemma~\ref{VOLT},
it remains to check that if~$u$ is defined as in~\eqref{a.asj},
then~$u\in C_b^{k,\alpha}$ and it satisfies~\eqref{amam}.
To check these facts, we first recall Lemma~\ref{ANCD}, according to which
$u\in AC^{k-1}([b,+\infty))$, and, for any~$t>b$,
\begin{eqnarray*}u^{(k)}(t)&=&\frac1{\Gamma(\alpha)}\Bigg[
\int_b^t g^{(k)}(\tau)(t-\tau)^{\alpha-1}d\tau\\&&\qquad +
\sum_{i=0}^{k-1} 
\frac{(\alpha+i)\dots(\alpha+i-k+1)\,g^{(i)}(b)}{\alpha(\alpha+1)\dots(\alpha+i)}
\;(t-b)^{\alpha+i-k}\Bigg]
.\end{eqnarray*}
Therefore, in the notation of~\eqref{CkAM1},
\begin{equation}\label{HANASKS2}\begin{split}\Theta_{k,\alpha,u,t}(\sigma)\,&=
u^{(k)}(\sigma)(t-\sigma)^{k-\alpha-1}\\&=
\frac1{\Gamma(\alpha)}\Bigg[
\int_b^\sigma g^{(k)}(\tau)(\sigma-\tau)^{\alpha-1}(t-\sigma)^{k-\alpha-1}\,d\tau\\&\qquad+
\sum_{i=0}^{k-1}
\frac{(\alpha+i)\dots(\alpha+i-k+1)\,g^{(i)}(b)}{\alpha(\alpha+1)\dots(\alpha+i)}
\;(\sigma-b)^{\alpha+i-k}(t-\sigma)^{k-\alpha-1}\Bigg]
.\end{split}\end{equation}
We observe that
\begin{equation}\label{HANASKS}\begin{split}&
\int_b^t\left|\sum_{i=0}^{k-1}\frac{(\alpha+i)\dots(\alpha+i-k+1)\,g^{(i)}(b)}{
\alpha(\alpha+1)\dots(\alpha+i)}
\;(\sigma-b)^{\alpha+i-k}(t-\sigma)^{k-\alpha-1}\right|\,d\sigma\\
\le\;& \sum_{i=0}^{k-1} \left|\frac{(\alpha+i)\dots(\alpha+i-k+1)\,g^{(i)}(b)}{
\alpha(\alpha+1)\dots(\alpha+i)}\right|\,\int_b^t
(\sigma-b)^{\alpha+i-k}(t-\sigma)^{k-\alpha-1}\,d\sigma\\=\;&
\sum_{i=0}^{k-1} \left|\frac{(\alpha+i)\dots(\alpha+i-k+1)\,g^{(i)}(b)}{
\alpha(\alpha+1)\dots(\alpha+i)}\right|\,{\mathrm{B}}(\alpha+i-k+1,k-\alpha)\,(t-b)^i,
\end{split}\end{equation}
where~\eqref{BETA} has been used in the last line (with~$x:=\alpha+i-k+1$ and~$y:=k-\alpha$).

On the other hand, making again use of~\eqref{BETA} with~$x:=\alpha$ and~$y:=k-\alpha$ here,
we see that
\begin{eqnarray*}
&&\int_b^t\left|
\int_b^\sigma g^{(k)}(\tau)(\sigma-\tau)^{\alpha-1}(t-\sigma)^{k-\alpha-1}\,d\tau\right|\,d\sigma\\
&\le&\int_b^t\left[
\int_b^\sigma \big|g^{(k)}(\tau)\big|\,
(\sigma-\tau)^{\alpha-1}(t-\sigma)^{k-\alpha-1}\,d\tau\right]\,d\sigma\\
&\le&\int_b^t\big|g^{(k)}(\tau)\big|\,\left[
\int_\tau^t 
(\sigma-\tau)^{\alpha-1}(t-\sigma)^{k-\alpha-1}\,d\sigma\right]\,d\tau\\
&=&{\mathrm{B}}(\alpha,k-\alpha)\,\int_b^t 
\big|g^{(k)}(\tau)\big|\,(t-\tau)^{k-1}\,d\tau,
\end{eqnarray*}
which is finite, thanks to our assumptions on~$g$.

Plugging this estimate and~\eqref{HANASKS} into~\eqref{HANASKS2}, we thereby deduce that,
for all~$t>b$,
$$ \int_b^t |\Theta_{k,\alpha,u,t}(\sigma)|\,d\sigma<+\infty,$$
and therefore~$u\in C_b^{k,\alpha}$.

With this, it only remains to check~\eqref{amam}. To this end,
we observe that the initial point conditions are satisfied, due to~\eqref{AKSMSxS}.
Moreover, \eqref{Cva} gives that
\begin{eqnarray*} \Gamma(\alpha)\,u^{(k)}(t)&=&
\int_b^t g^{(k)}(\sigma)\,(t-\sigma)^{\alpha-1}\,d\sigma\\&&\qquad
+
\sum_{i=0}^{k-1} 
\frac{(\alpha+i)\dots(\alpha+i-k+1)\,g^{(i)}(b)}{\alpha(\alpha+1)\dots(\alpha+i)}
\,(t-b)^{\alpha+i-k},
\end{eqnarray*}
and therefore, in view of~\eqref{defcap} and~\eqref{BETA},
\begin{eqnarray*}&&
\Gamma(\alpha)\Gamma(k-\alpha)\,D_b^{\alpha}u(t)\\&=&
\Gamma(\alpha)\,\int_b^t 
\frac{u^{(k)}\left(\tau\right)}{(t-\tau)^{\alpha-k+1}} d\tau
\\&=& \int_b^t\left(
\int_b^\tau g^{(k)}(\sigma)\,(\tau-\sigma)^{\alpha-1}\,d\sigma
\right)(t-\tau)^{k-\alpha-1}\,d\tau\\&&\qquad
+
\sum_{i=0}^{k-1} \int_b^t
\frac{(\alpha+i)\dots(\alpha+i-k+1)\,g^{(i)}(b)}{\alpha(\alpha+1)\dots(\alpha+i)}\,
(\tau-b)^{\alpha+i-k}(t-\tau)^{k-\alpha-1}\,d\tau
\\&=& \int_b^tg^{(k)}(\sigma)\,\left(
\int_\sigma^t (\tau-\sigma)^{\alpha-1}(t-\tau)^{k-\alpha-1}\,d\tau
\right)\,d\sigma\\&&\qquad
+
\sum_{i=0}^{k-1} {\mathrm{B}}(\alpha+i-k+1,k-\alpha)
\frac{(\alpha+i)\dots(\alpha+i-k+1)\,g^{(i)}(b)}{\alpha(\alpha+1)\dots(\alpha+i)}\,
(t-b)^i\\&=&{\mathrm{B}}(\alpha,k-\alpha)\,
\int_b^tg^{(k)}(\sigma)\,(t-\sigma)^{k-1}d\sigma\\&&\qquad
+
\sum_{i=0}^{k-1} {\mathrm{B}}(\alpha+i-k+1,k-\alpha)
\frac{(\alpha+i)\dots(\alpha+i-k+1)\,g^{(i)}(b)}{\alpha(\alpha+1)\dots(\alpha+i)}\,
(t-b)^i.
\end{eqnarray*}
Hence, recalling~\eqref{bega} and using the fact that~$\Gamma(z+1)=z\Gamma(z)$, we have that
\begin{eqnarray*}
&&\Gamma(\alpha)\Gamma(k-\alpha)\,D_b^{\alpha}u(t)\\&=&
\frac{\Gamma(\alpha)\Gamma(k-\alpha)}{(k-1)!}\,
\int_b^tg^{(k)}(\sigma)\,(t-\sigma)^{k-1}d\sigma\\&&\qquad
+
\sum_{i=0}^{k-1} 
\frac{\Gamma(\alpha+i-k+1)\Gamma(k-\alpha)}{i!}\cdot
\frac{(\alpha+i)\dots(\alpha+i-k+1)\,g^{(i)}(b)}{\alpha(\alpha+1)\dots(\alpha+i)}\,
(t-b)^i
\\&=&
\frac{\Gamma(\alpha)\Gamma(k-\alpha)}{(k-1)!}\,
\int_b^tg^{(k)}(\sigma)\,(t-\sigma)^{k-1}d\sigma
+\Gamma(\alpha)\Gamma(k-\alpha)\,
\sum_{i=0}^{k-1}\frac{g^{(i)}(b)}{i!}\,(t-b)^i
.
\end{eqnarray*}
As a consequence, recalling~\eqref{76},
$$ \Gamma(\alpha)\Gamma(k-\alpha)\,D_b^{\alpha}u(t)=
\Gamma(\alpha)\Gamma(k-\alpha)\,g(t),$$
which gives that~$D_b^{\alpha}u(t)=g(t)$, as desired.
\end{proof}
\end{lemma}

A bootstrap regularity theory for time-fractional equations
leads to additional smoothness of the solution. In our framework, the result needed is the following:

\begin{lemma}
\label{morereg}
Let $g\in C^h([b,+\infty))$ for every~$h\in\mathbb{N}_0$,
and $u\in C^{k,\alpha}_b$ be a solution of
\begin{equation*}
\begin{cases}
D_b^\alpha u(t)=g(t)&\quad\text{in}\quad(b,+\infty), \\
u^{(h)}(b)=0&\quad{\mbox{ for any }}\quad h=0,\ldots,k-1,
\end{cases}
\end{equation*}
Then $u\in C^\infty((b,+\infty))$.
\begin{proof}
In light of~\eqref{a.asj}, we can write, for every~$t>b$,
$$ u(t)=\frac{1}{\Gamma(\alpha)}
\int_0^{t-b} g(t-\sigma)\sigma^{\alpha-1} d\sigma.
$$
The desired result follows by taking derivatives in~$t$.
\end{proof}
\end{lemma}

\subsection{Existence of a sequence of Caputo-stationary functions that tends to the function $t^\alpha$}

Now, we generalize some
results contained in Section 3 of \cite{MR3716924} concerning the boundary
asymptotics of solutions of fractional equations, and we
construct a sequence of Caputo-stationary functions which tends 
to the function $t^\alpha$
uniformly
on bounded subintervals of~$(0,+\infty)$. Differently from the previous literature, we deal
with fractional derivatives of any order.

More precisely, the result that we need is the following:

\begin{lemma}
\label{cat}
Let~$b\in(-\infty,0]\cup\{-\infty\}$.
Let $\psi_0\in C^{k,\alpha}_{-\infty}$ be such that
\begin{equation}\label{EQNSMS}
\psi_0^{(k)}=0\quad\text{for any}\quad t\in\left(-\infty, 0\right),\qquad{\mbox{and}}\qquad 
\psi_0(t)=0\quad\text{for any}\quad t\in
\left[\frac{3}{4},1\right] .\end{equation}
Then, there exists $\psi\in C_b^{k,\alpha}$ such that
\begin{equation}
\label{vpv}
\begin{cases}
D_b^\alpha\psi(t)=0\quad&\text{in}\quad (1,+\infty) ,\\
\psi(t)=\psi_0(t)\quad&\text{in}\quad (-\infty,1].
\end{cases}
\end{equation}
Moreover, 
$\psi\in C^\infty((1,+\infty))$, and we have that
\begin{equation}
\label{asycap}
\psi(1+\epsilon)=\kappa\epsilon^\alpha+o(\epsilon^{\alpha}),
\end{equation}
as $\epsilon\to 0^+$, 
with
\begin{equation}
\label{asycap2}
\kappa:=-\frac{1}{\Gamma(\alpha)\Gamma(k-\alpha)}\,
\int_0^1\left(\int_0^{3/4}\psi_0^{(k)}(\omega)(1-\omega)^{k-\alpha-1}
d\omega\right)(1-z)^{\alpha-1}dz\in\mathbb{R}.
\end{equation}

\begin{proof}
For every $t\in[1,+\infty)$, we set 
\begin{equation}\label{DBEgAB}
g(t):=-\frac{1}{\Gamma(k-\alpha)}\,\int_0^{3/4}\psi_0^{(k)}(\tau)(t-\tau)^{k-\alpha-1}d\tau.\end{equation}
By construction, the term~$t-\tau$ in the integrand above never vanishes,
thus permitting to take derivatives inside the integral sign. Consequently,
we have that~$g\in C^h([1,+\infty))$ for every~$h\in\mathbb{N}_0$.

For any~$t>1$, we define
\begin{equation}\label{WE SBKD}
\psi(t):=\frac{1}{\Gamma(\alpha)}
\int_1^t g(\tau)(t-\tau)^{\alpha-1} d\tau.
\end{equation}
We know from Lemma~\ref{VOLT} that
\begin{equation}\label{ALsloldfffaaaaDF}
\psi\in C^{k,\alpha}_1\end{equation} is a solution of
\begin{equation}\label{7uA923cv IK} \begin{cases}
D_1^\alpha \psi(t)=g(t)&\quad\text{in}\quad(1,+\infty) ,\\
\psi^{(h)}(1)=0&\quad{\mbox{for any}}\quad h=0,\ldots,k-1,
\end{cases}\end{equation}
and also~$\psi\in C^\infty((1,+\infty))$,
due to Lemma~\ref{morereg}.

We also extend~$\psi$ in~$(-\infty,1]$ by setting~$\psi(t):=\psi_0(t)$
for all~$t\in(-\infty,1]$. Since $\psi_0$ vanishes in~$\left[\frac34,1\right]$,
using the initial condition in~\eqref{7uA923cv IK} we have that~$ \psi^{(j)}_0(1)=0=\psi^{(j)}(1)$
for each~$j\in\{1,\dots,k-1\}$. Thus,
recalling~\eqref{ALsloldfffaaaaDF}
and using Lemma~\ref{LA2lem} (exploited here with~$f:=\psi_0$
and~$g:=\psi$), we find that
\begin{equation}\label{Bam}
\psi\in C^{k,\alpha}_b.
\end{equation}
Then, in light of~\eqref{7uA923cv IK}, we can write that
\begin{equation}\label{7uAH789} \begin{cases}
D_1^\alpha \psi(t)=g(t)&\quad\text{in}\quad(1,+\infty) ,\\
\psi=\psi_0 &\quad\text{in}\quad(-\infty,1].
\end{cases}\end{equation}
{F}rom~\eqref{Bam}, it follows in particular that~$\psi\in C^{k,\alpha}_0$.
Consequently, by~\eqref{7uAH789} and Lemma \ref{equiv},
we have that
\begin{equation}\label{SPL}
{\mbox{$\psi$ is a solution of~\eqref{vpv} with~$b=0$.}}\end{equation}
Since~$\psi(t)=\psi_0(t)=\psi_0(0)$ if~$t\in (-\infty,0]$, we see that~$\psi^{(k)}(\tau)=0$
in~$(-\infty,0)$ and therefore, for every~$t\in(1,+\infty)$,
$$ D^\alpha_b\psi(t)=
\frac{1}{\Gamma(k-\alpha)}
\int_b^t \frac{\psi^{(k)}\left(\tau\right)}{(t-\tau)^{\alpha-k+1}} d\tau=
\frac{1}{\Gamma(k-\alpha)}
\int_0^t \frac{\psi^{(k)}\left(\tau\right)}{(t-\tau)^{\alpha-k+1}} d\tau=
D^\alpha_0\psi(t)=0,$$
thanks to~\eqref{SPL}
and this gives~\eqref{vpv} (alternatively, one can use
Lemma~\ref{uNA.AOKa} here).

Hence, to complete the proof of Lemma~\ref{cat}, it only remains to establish~\eqref{asycap}.
For this, let $\epsilon>0$ and $t:=1+\epsilon$. Then, by \eqref{WE SBKD},
$$
\Gamma(\alpha)\,\psi(1+\epsilon)=\int_1^{1+\epsilon} g(\tau)(1+\epsilon-\tau)^{\alpha-1}d\tau=\epsilon^\alpha\int_0^1 g(\epsilon z+1)(1-z)^{\alpha-1} dz,
$$
where the change of variables $\tau=\epsilon z+1$ has been used.

Furthermore, by~\eqref{DBEgAB},
$$
g(\epsilon z+1)=-\frac{1}{\Gamma(k-\alpha)}\,
\int_0^{3/4}\psi_0^{(k)}(\omega)(\epsilon z+1-\omega)^{k-\alpha-1}d\omega.
$$
Hence
$$
\Gamma(\alpha)\,\psi(1+\epsilon)=-\frac{\epsilon^\alpha}{\Gamma(k-\alpha)}\,
\int_0^1\left(\int_0^{3/4}\psi_0^{(k)}(\omega)(\epsilon z+1-\omega)^{k-\alpha-1}d\omega\right)(1-z)^{\alpha-1}dz.
$$
This gives that
$$ \lim_{\epsilon\to0^+}\epsilon^{-\alpha}
\Gamma(\alpha)\Gamma(k-\alpha)\psi(1+\epsilon)=-
\int_0^1\left(\int_0^{3/4}\psi_0^{(k)}(\omega)(1-\omega)^{k-\alpha-1}d\omega\right)(1-z)^{\alpha-1}dz.$$
This, together with~\eqref{asycap2}, establishes~\eqref{asycap},
as desired.
\end{proof}
\end{lemma}

In our setting, it is crucial that we can choose~$\psi_0$ such that~$\kappa$
in~\eqref{asycap2} is not zero. This is warranted by the following observation:

\begin{lemma} \label{SEGNP}
There exists~$\psi_0$ satisfying all the assumptions of Lemma~\ref{cat}
and such that~$\kappa>0$,
where the setting in~\eqref{asycap2}
has been used.

\begin{proof} We let
$$\psi_0(t):=\begin{cases}
(-1)^{k-1}
\displaystyle\sum_{j=0}^{k-1}\left({k}\atop{j}\right)
\left(\displaystyle\frac34\right)^{k-j}t^j & {\mbox{ if }}t\le 0,\\
(-1)^{k-1}\left(\displaystyle\frac34-t\right)^k &{\mbox{ if }}t\in\left(0,\displaystyle\frac34\right),\\
0 & {\mbox{ if }}t\ge\displaystyle\frac34.             
\end{cases}$$
We observe that the statements in~\eqref{EQNSMS} are satisfied in this case.
Furthermore, we claim that
\begin{equation}\label{8JANNasadf}
\psi_0\in C^{k,\alpha}_{-\infty}.\end{equation}
Indeed, using Lemma~\ref{LA2lem} with
$$f(t):=(-1)^{k-1}\left(\displaystyle\frac34-t\right)^k,$$
$g:=0$, $a:=0$ and~$b:=\frac34$,
we obtain that~$\tilde\psi_0\in C^{k,\alpha}_{0}$, where
$$ \tilde\psi_0(t):=\begin{cases}
(-1)^{k-1}\left(\displaystyle\frac34-t\right)^k &{\mbox{ if }}t\in\left(0,\displaystyle\frac34\right),\\
0 & {\mbox{ if }}t\ge\displaystyle\frac34.             
\end{cases}$$
Then, using again
Lemma~\ref{LA2lem} with
$$f(t):=(-1)^{k-1}
\displaystyle\sum_{j=0}^{k-1}\left({k}\atop{j}\right)
\left(\displaystyle\frac34\right)^{k-j}t^j=(-1)^{k-1}\left(\displaystyle\frac34-t\right)^k-(-1)^{k-1}t^k
,$$
$g:=\tilde\psi_0$, $a:=-\infty$ and~$b:=0$,
we obtain~\eqref{8JANNasadf}, as desired.

We also notice that, for any~$t\in\left(0,\frac34\right)$, we have that
$$ \psi_0^{(k)}(t)= -k!<0,$$
and so, recalling~\eqref{asycap2}, we have that~$\kappa>0$.
\end{proof}
\end{lemma}

Now we point out that the function built in Lemmata \ref{cat}
and~\ref{SEGNP} can be conveniently rescaled, taking advantage of the scaling invariance
of the operator, and in this way one can single out the boundary behavior.
Namely, we have that:

\begin{lemma}
\label{lemblo}
There exists a sequence $(v_j)_{j\in\mathbb{N}}$ of functions $v_j\in C_{-\infty}^{k,\alpha}\cap C^\infty((0,+\infty))$ such that, for any $j\in\mathbb{N}$, $v_j$ solves the following problem
\begin{equation}
\label{problow}
\begin{cases}
D^\alpha_{-\infty}v_j(t)=0&\quad\text{in}\quad(0,+\infty), \\
v_j(t)=0&\quad\text{in}\quad\left[-\frac{j}{4},0\right],
\end{cases}
\end{equation}
and, for any $t>0$, 
\begin{equation}\label{asuvdfasdf}
\lim_{j\to+\infty}v_j(t)=\kappa t^\alpha,\end{equation} for some~$\kappa>0$,
and the convergence is uniform on any bounded subinterval of $(0,+\infty)$.

In addition, 
\begin{equation}\label{acostPA}
v^{(k)}_j=0\,\mbox{ in }\,(-\infty,-j).
\end{equation}

\begin{proof}
Let $\psi$ be the 
function given in Lemmata \ref{cat}
and~\ref{SEGNP}, used here with~$b:=-\infty$,
and define for any $j\in\mathbb{N}$
\begin{equation*}
v_j(t):=j^\alpha\psi\left(\frac{t}{j}+1\right).
\end{equation*}
Since $\psi\in C_{-\infty}^{k,\alpha}\cap C^\infty((1,+\infty))$,
we have that $v_j\in C_{-\infty}^{k,\alpha}\cap C^\infty((0,+\infty))$. 

We claim that for any $j\in\mathbb{N}$, $v_j$ solves \eqref{problow}. 
Indeed, recalling~\eqref{EQNSMS} and~\eqref{vpv}, for any~$t\in\left[-\frac{j}4,0\right]$
we have that~$\frac{t}{j}+1\in\left[\frac{3}4,1\right]$, and
$$ v_j(t)=j^\alpha\psi\left(\frac{t}{j}+1\right)=j^\alpha\psi_0\left(\frac{t}{j}+1\right)=0.
$$
Moreover, if~$t>0$, we have that~$\frac{t}{j}+1>1$ and therefore,
using the change of variables~$y:=\frac{\tau}{j}+1$ and~\eqref{vpv},
\begin{equation*}
\begin{split}
D^\alpha_{-\infty}v_j(t)&=\frac{1}{\Gamma(k-\alpha)}\int_{-\infty}^t v_j^{(k)}(\tau)(t-\tau)^{k-\alpha-1}d\tau \\
&=\frac{j^{\alpha-k}}{\Gamma(k-\alpha)}\int_{-\infty}^t\psi^{(k)}\left(\frac{\tau}{j}+1\right)(t-\tau)^{k-\alpha-1}d\tau
\\&=\frac{j^{\alpha-k+1}}{\Gamma(k-\alpha)}\int_{-\infty}^{\frac{t}{j}+1}\psi^{(k)}(y)\left(t-j(y-1)\right)^{k-\alpha-1}dy \\
&=\frac{1}{\Gamma(k-\alpha)}\int_{-\infty}^{\frac{t}{j}+1}\psi^{(k)}(y)
\left(\frac{t}{j}+1-y\right)^{k-\alpha-1}dy \\&=D_{-\infty}^\alpha\psi\left(\frac{t}{j}+1\right)\\&=0.
\end{split}
\end{equation*}
This proves~\eqref{problow}.

Now, let~$I$ be a bounded subinterval of~$(0,+\infty)$.
Using formula \eqref{asycap}, for $t>0$ and for large $j$, we have that
\begin{eqnarray*}&&
\sup_{t\in I}|v_j(t)-\kappa t^\alpha|=
\sup_{t\in I}\left|
j^\alpha\psi\left(\frac{t}{j}+1\right)
-\kappa t^\alpha\right|
=\sup_{t\in I}\left|
j^\alpha\left(\kappa\frac{t^\alpha}{j^\alpha}+o\left(\frac{t^\alpha}{j^\alpha}\right)\right)
-\kappa t^\alpha\right|\\&&\qquad
=\sup_{t\in I}
j^\alpha\,o\left(\frac{t^\alpha}{j^\alpha}\right)=
j^\alpha\,o\left(\frac{1}{j^\alpha}\right)=o(1),
\end{eqnarray*}
thus proving the desired asymptotics in~\eqref{asuvdfasdf},
and~$\kappa>0$ here in view of Lemma~\ref{SEGNP}.

Finally, recalling~\eqref{EQNSMS} and~\eqref{vpv}, we have that if~$t<-j$
$$ v_j^{(k)}(t)=j^{\alpha-k}\psi^{(k)}\left(\frac{t}{j}+1\right)=
j^{\alpha-k}\psi^{(k)}_0\left(\frac{t}{j}+1\right)=0,$$
and this proves \eqref{acostPA}.
\end{proof}
\end{lemma}

\subsection{Maximal span property and proof of Theorem~\ref{dens}}

We now exploit a method introduced in~\cite{MR3626547}
and we take advantage of the boundary asymptotics established
in Lemma~\ref{SEGNP} to construct solutions of linear time-fractional
equations with a prescribed jet of derivatives at a point. {F}rom this,
the proof of Theorem~\ref{dens} will be completed, by polynomial approximation
and scaling. In our strategy is also technically more convenient to
prove a slightly different modified version of Theorem~\ref{dens},
in which the initial point in which the Caputo derivative is~$-\infty$
and the approximating function is constant near~$-\infty$. Namely, we prove the following result:

\begin{theorem}
\label{dens2}
Let $h\in\mathbb{N}_0$, $k\in\mathbb{N}$, and~$\alpha\in(k-1,k)$. 

Then, for every $f\in C^h([0,1])$ and $\epsilon>0$, there exist $a<0$
and $u\in C^{k,\alpha}_{-\infty}$ such that 
\begin{eqnarray}
\label{CLAIM12}&&D_{-\infty}^\alpha u(t)=0\quad\text{in}\quad[0,+\infty)\\
\label{CLAIMADDP} && u^{(k)}=0\quad\text{for all }t\in(-\infty,a]\\
\label{CLAIM22}{\mbox{and }}&&\left\|u-f\right\|_{C^h([0,1])}<\epsilon.
\end{eqnarray}
\end{theorem}

By Lemma~\ref{AfavaadfeAU}, we observe that
\begin{equation}\label{IMPLIES}
{\mbox{Theorem~\ref{dens2} implies Theorem~\ref{dens}.}}
\end{equation}
Hence, in light of~\eqref{IMPLIES}, to prove Theorem~\ref{dens}
we will focus on the proof of Theorem~\ref{dens2}. For this,
one of the crucial arguments is given by the following ``cherry picking''
result:

\begin{proposition}
\label{maxcapspan}
For any $m\in\mathbb{N}$, there exist~$p>0$, $R>0$, and~$v\in C_{-\infty}^{k,\alpha}\cap C^\infty((0,+\infty))$ such that
\begin{equation}\label{UHAASasdIDKF}
\begin{cases}
D^\alpha_{-\infty}v(t)=0&\quad\text{for all}\quad t\in(0,+\infty), \\
v^{(k)}=0&\quad\text{in}\quad (-\infty,-R),
\end{cases}
\end{equation}
\begin{equation*}
v^{(l)}(p)=0\quad{\mbox{ for any }}\; l\in\{0,\dots,m-1\},\end{equation*} 
and 
\begin{equation*}
v^{(m)}(p)=1.\end{equation*}

\begin{proof}
Let~$\mathcal{Z}_0$ be the set containing all the functions~$v
\in C_{-\infty}^{k,\alpha}\cap
C^\infty((0,+\infty))$ such that
\begin{equation}\label{v1v1}
D^\alpha_{-\infty}v=0\quad\text{in}\quad(0,+\infty),\end{equation}
and for which there exists~$R>0$ such that
\begin{equation}
\label{v1v2}
v^{(k)}=0\quad\text{in}\quad(-\infty,-R).\end{equation}
Let also~$\mathcal{Z}:=\mathcal{Z}_0
\times(0,+\infty)$.

To each pair $(v,t)\in\mathcal{Z}$ we associate the vector
$(v(t),v'(t),\ldots,v^{(m)}(t))\in\mathbb{R}^{m+1}$ and 
consider~$\mathcal{V}$ to be the set
\begin{equation}\label{JASNsIOKA} \mathcal{V}:=\big\{(v(t),v'(t),\ldots,v^{(m)}(t)),
\quad{\mbox{ with }}\; (v,t)\in\mathcal{Z}
\big\}.\end{equation}
We point out that
\begin{equation}\label{VbasJA}
{\mbox{$\mathcal{V}$ is a vector space.}}\end{equation}
Notice indeed that if~$(v_1,t)$, $(v_2,t)\in\mathcal{Z}$
and~$\lambda_1$, $\lambda_2\in\mathbb{R}$, we have that~$v_i\in C_{-\infty}^{k,\alpha}\cap
C^\infty((0,+\infty))$,
and therefore~$v_*:=\lambda_1 v_1+\lambda_2 v_2\in
C_{-\infty}^{k,\alpha}\cap
C^\infty((0,+\infty))$.
Also, $v_*$ satisfies~\eqref{v1v1} by linearity of the operator~$D^\alpha_{-\infty}$.
In addition, by~\eqref{v1v2}, for each~$i\in\{1,2\}$
we know that~$v^{(k)}_i=0$ in~$(-\infty,-R_i)$
for some~$R_i>0$, and therefore $v_*$ satisfies~\eqref{v1v2}
with~$R:=\max\{R_1,R_2\}>0$. This completes the proof of~\eqref{VbasJA}.

Now, we claim that
\begin{equation}\label{CLSPA}
\mathcal{V}=\mathbb{R}^{m+1}.\end{equation}
To check this, we suppose by contradiction that~$\mathcal{V}$ lies in a 
proper subspace of~$\mathbb{R}^{m+1}$. Then, by~\eqref{VbasJA},
$\mathcal{V}$ must lie in a
hyperplane, hence there exists 
\begin{equation}
\label{cnonull}
(c_0,\ldots,c_m)\in\mathbb{R}^{m+1}\setminus\left\{0\right\} 
\end{equation}
which is orthogonal to any vector $(v(t),\ldots,v^{(m)}(t))$ with $(v,t)\in\mathcal{Z}$, 
namely
\begin{equation}
\label{prp}
\sum_{i=0}^m c_i v^{(i)}(t)=0.
\end{equation}
We notice that for any $j\geq 1$ the pair $(v_j,t)$, with $v_j$ satisfying \eqref{problow}
and~\eqref{acostPA},
and $t\in(0,+\infty)$, belongs to~$\mathcal{Z}$. Consequently, writing~\eqref{prp}
in this case, it follows that, for any~$j\geq 1$,
\begin{equation}
\label{prp2}
\sum_{i=0}^m c_i v_j^{(i)}(t)=0.
\end{equation}
Let now~$\varphi\in C_c^\infty((0,+\infty))$. 
Integrating by parts,
by Lemma \ref{lemblo} and the Dominated Convergence Theorem, 
we have that, for any $i\in\mathbb{N}$,
\begin{eqnarray*}
&&\lim_{j\to+\infty}\int_{-\infty}^{+\infty}v_j^{(i)}(t)\varphi(t)dt
=(-1)^i\lim_{j\to+\infty}\int_{-\infty}^{+\infty}v_j(t)\varphi^{(i)}(t)dt\\
&&\qquad=(-1)^i\int_{-\infty}^{+\infty}\kappa t^\alpha\varphi^{(i)}(t)dt
=\kappa\int_{-\infty}^{+\infty}(t^\alpha)^{(i)}\varphi(t)dt.\end{eqnarray*}
Multiplying by $c_i$ and summing up, 
recalling also~\eqref{prp2},
we thereby obtain that
\[0=\lim_{j\to+\infty}\int_{-\infty}^{+\infty}\sum_{i=0}^m c_i v_j^{(i)}(t)
\varphi(t)dt=\kappa\int_{-\infty}^{+\infty}\sum_{i=0}^m c_i(t^\alpha)^{(i)}\varphi(t)dt,\] 
for any $\varphi\in C_c^\infty((0,+\infty))$. 

This gives that, for every $t\in(0,+\infty)$,
\[0=\kappa\sum_{i=0}^m c_i(t^\alpha)^{(i)}=\kappa\sum_{i=0}^m c_i
\alpha(\alpha-1)\ldots(\alpha-i+1)t^{\alpha-i}.\] 
Then, we divide this relation by $\kappa>0$ and multiply by $t^{m-\alpha}$: 
in this way we obtain that,
for every $t\in(0,+\infty)$,
\[\sum_{i=0}^m c_i\alpha(\alpha-1)\ldots(\alpha-i+1)t^{m-i}=0.\] 
The identity above describes
a polynomial which vanishes for any $t\in(0,+\infty)$. As a result, the Identity Principle
for polynomials leads to
$$ c_i\alpha(\alpha-1)\ldots(\alpha-i+1)=0,$$
for all~$i\in\{0,\dots,m\}$.

Consequently, since $\alpha\in\mathbb{R}\setminus\mathbb{N}$, 
the product $\alpha(\alpha-1)\ldots(\alpha-i+1)$ never vanishes, 
and so the coefficients $c_i$ are forced to be null for any $i\in\{0,\dots,m\}$. 
This is in contradiction with~\eqref{cnonull}, and therefore the proof
of~\eqref{CLSPA} is complete.

{F}rom this, the desired claim in
Proposition~\ref{maxcapspan} plainly follows.
\end{proof}
\end{proposition}

Now, we use the function constructed in Proposition \ref{maxcapspan} to prove Theorem \ref{dens}. 

\begin{proof}[Proof of Theorem \ref{dens}]
In light of~\eqref{IMPLIES}, we focus on the proof of
Theorem \ref{dens2}.
Also,
it is sufficient to prove Theorem \ref{dens2} for monomials:
indeed, once proved in this case, the claim is true for polynomials simply
by linearity of the operator~$D^\alpha_{-\infty}$, and then
it is also true for smooth functions, by exploiting the density
of polynomials in the space $C^h([0,1])$ given by the Stone-Weierstra{\ss} Theorem.

Therefore, in place of the generic function~$f$ in Theorem \ref{dens2},
we can consider the monomial~$q_m(t):=\frac{t^m}{m!}$.
We take~$p$, $R>0$, and~$v$ as in Proposition \ref{maxcapspan}.
Also, let
\begin{equation}\label{deltach}
\delta\in(0,1)\end{equation}
to be chosen conveniently in the sequel. 

Let $u$ be the function \begin{equation}\label{gAJAAAsfdf}
u(t):=\frac{v(\delta t+p)}{\delta^m}.\end{equation}
By construction, we have that~$u\in C^{k,\alpha}_{-\infty}\cap
C^\infty\left(\left(-\frac{p}{\delta},+\infty\right)\right)$, and, for any~$t\in\left(-\frac{p}{\delta},+\infty\right)$,
\begin{equation*}
\begin{split}
\Gamma(k-\alpha)D^{\alpha}_{-\infty}u(t)&=\int_{-\infty}^t u^{(k)}(\tau)(t-\tau)^{k-\alpha-1}d\tau \\
&=\delta^{k-m}\int_{-\infty}^t v^{(k)}(\delta\tau+p)(t-\tau)^{k-\alpha-1}d\tau \\
&=\delta^{k-m}\delta^{\alpha-k+1}\int_{-\infty}^{\delta t+p}v^{(k)}(y)(\delta t+p-y)^{k-\alpha-1}\frac{dy}{\delta} \\
&=\delta^{\alpha-m}\Gamma(k-\alpha)D^\alpha_{-\infty}v(\delta t+p)
\\&=0.
\end{split}
\end{equation*}
Then,
we have that $u\in C^{k,\alpha}_{-\infty}$, and since $-\frac{p}{\delta}<0$, we also see that~$
D^\alpha_{-\infty}u(t)=0$ in $[0,+\infty)$.

This proves~\eqref{CLAIM12}. In addition, we see that~\eqref{CLAIMADDP} 
follows from~\eqref{UHAASasdIDKF} and~\eqref{gAJAAAsfdf}
(taking~$a:=\frac{-R-p}{\delta}$).

Hence, we now focus on the proof of~\eqref{CLAIM22}.
For this, we observe that
\[u^{(l)}(0)=\delta^{l-m}v^{(l)}(p)=0,\qquad\quad{\mbox{for any }}\; l\in\{0,\dots,m-1\}\]
and \[u^{(m)}(0)=v^{(m)}(p)=1.\]
Now, for any $t>-\frac{p}{\delta}$, we set \[g(t):=u(t)-q_m(t).\] 
We have that 
\begin{equation}
\label{preder}
\begin{split}
&g^{(l)}(0)=0,\qquad\quad{\mbox{for any }}\; l\in\{0,\dots, m\},\qquad\quad\text{and} \\
&g^{(m+l)}(t)=u^{(m+l)}(t)=\delta^l v^{(m+l)}(\delta t+p),
\qquad\quad{\mbox{for any }}\; l\in\{ 1,2,\dots\}.
\end{split}
\end{equation}
Hence, for all~$t\in[0,1]$ and~$l\in\{ 1,2,\dots\}$,
\begin{equation}
\label{cindu}
|g^{(m+l)}(t)|\leq \delta^l\sup_{y\in[p,p+\delta]}|v^{(m+l)}(y)|=\tilde{c}\delta^l,
\end{equation}
where $\tilde{c}$ is a positive constant, depending on~$v$, $m$ and~$l$.

Now, we consider the derivative of order $j\in\{0,\dots,h\}$ of $g$
(with the notation that the derivative of order zero coincides with the function itself),
and we take its Taylor expansion with Lagrange remainder. 
In this way,
in view of \eqref{preder},
we have that, for every~$t\in(0,1)$,
\[g^{(j)}(t)=\sum_{i=\max\left\{j,m+1\right\}}^{j+m+1}g^{(i)}(0)
\frac{t^{i-j}}{(i-j)!}+g^{(m+h+2)}(c)\frac{t^{m+2}}{(m+2)!},\] 
for some $c\in(0,t)$ possibly depending on~$j$, $m$ and~$t$.

As a consequence,
using \eqref{deltach} and~\eqref{cindu}, and possibly
renaming the constants, we obtain that, for any $t\in[0,1]$,
\[ |q_m^{(j)}(t)-u^{(j)}(t)|=|g^{(j)}(t)|\leq C\delta,\]
with~$C>0$ possibly depending on~$j$ and~$m$.
By summing this inequality over~$j\in\{1,\dots,h\}$,
and by choosing~$\delta$ sufficiently small with respect to~$\epsilon$,
we complete the proof of~\eqref{CLAIM22}.\end{proof}

\section{Proof of Corollary~\ref{densepsi}}\label{9ikHNA:SP2}

The proof of Corollary~\ref{densepsi} relies on Theorem~\ref{dens}
and on a change of variable induced by the function~$\psi$.
For this, we make the following observation:

\begin{lemma}
\label{capty}
Let $a$, $\alpha$ and $\psi$ be as in~\eqref{caputotype}. Then,
for any $t>a$,
\begin{equation}
\label{ty}
D_a^{\alpha,\psi}u(t)=D_{\psi(a)}^\alpha (u\circ\psi^{-1})(\psi(t))
.\end{equation}
\begin{proof} 
Using the change of variable $\omega:=\psi(\tau)$,
we have that $\psi'(\tau)\,d\tau
=d\omega$, and $\frac{1}{\psi'(\tau)}\frac{d}{d\tau}=
\frac{d}{d\omega}$. Thus, from~\eqref{caputotype} we see that
\[
D_a^{\alpha,\psi}u(t)=\frac{1}{\Gamma\left(k-\alpha\right)}
\int_{\psi(a)}^{\psi(t)} {\frac{(u\circ\psi^{-1})^{(k)}(\omega)}{
(\psi(t)-\omega)^{\alpha-k+1}}\, d\omega}=D_{\psi(a)}^\alpha (u\circ\psi^{-1})(\psi(t)),
\]
and this proves \eqref{ty}.
\end{proof}
\end{lemma}

Now we can complete the proof of Corollary~\ref{densepsi}.

\begin{proof}[Proof of Corollary~\ref{densepsi}]
We claim that
\begin{equation}\label{SUS}
{\mbox{for every~$\lambda\in \big(-\infty,\psi(0)\big)$ there exists~$t_\lambda\in(-\infty,0)$
such that~$\psi(t_\lambda)=\lambda$.}}
\end{equation}
To check this, we exploit~\eqref{CPSI} to find~$\theta_\lambda<0$ such that~$\psi(\theta_\lambda)<\lambda$.
Since~$\psi(0)>\lambda$, the result in~\eqref{SUS} follows from the Mean Value Theorem.

Now, for any~$\omega\in[\psi(0),\psi(1)]$
we let~$\tilde f(\omega):=f\big(\psi^{-1}(\omega)\big)$.
Notice that~$\tilde f\in C^h([\psi(0),\psi(1)])$.
Hence, in light of Theorem \ref{dens}, we find~$\tilde a\in\big(-\infty,\psi(0)\big)$
and~$\tilde u\in C^{k,\alpha}_{\tilde a}$ such that
\begin{eqnarray*}
&&D_{\tilde a}^\alpha \tilde u(\omega)=0\quad\text{for every}\quad\omega\in[\psi(0),+\infty)\\
{\mbox{and }}&&\big\|\tilde u-\tilde f\big\|_{C^h([\psi(0),\psi(1)])}<\epsilon.\end{eqnarray*}
In light of~\eqref{SUS}, 
there exists~$a\in(-\infty,0)$ such that
\begin{equation}\label{SUSIS}\psi(a)=\tilde a.
\end{equation}
Then, we set, for any~$t\in{\mathbb{R}}$ and~$\omega\in[\psi(0),\psi(1)]$,
\begin{eqnarray*} && u(t):=\tilde u\big(\psi(t)\big)\\ {\mbox{and }}
&& \tilde v(\omega):=\tilde u(\omega)-\tilde f(\omega)=
u\big(\psi^{-1}(\omega)\big)-
f\big(\psi^{-1}(\omega)\big).\end{eqnarray*}
When~$t\in[0,1]$, we also set
$$ v(t):=\tilde v\big(\psi(t)\big)=u(t)-f(t).$$
By the Fa\`a di Bruno Formula, for any~$j\in\{0,\dots,h\}$,
\begin{eqnarray*}
&& v^{(j)}(t)=\frac{d^j}{dt^j}\tilde v\big(\psi(t)\big)=
j!\sum _{m =1}^{j}{\frac {(D^{m}\tilde v)(\psi(t))}{m !}}
\sum _{{h_{1},\cdots ,h_{m }\ge1}\atop{h_{1}+\cdots +h_{m }=j}}
\frac{D^{h_{1}}\psi(t)}{h_{1}!}\cdots\frac {D^{h_{m }}\psi(t)}{h_{m}!}
\end{eqnarray*}
and therefore
\begin{eqnarray*}&&\| u-f \|_{C^h([0,1])}=
\| v \|_{C^h([0,1])}=\sum_{j=0}^h\sup_{t\in(0,1)}|v^{(j)}(t)|\\
&&\qquad\le h! \sum_{j=0}^h\sum _{m =1}^{j}\sup_{t\in(0,1)}
|(D^{m}\tilde v)(\psi(t))|\,
\sum _{{h_{1},\cdots ,h_{m }\ge1}\atop{h_{1}+\cdots +h_{m }=j}}
|D^{h_{1}}\psi(t)|\cdots|D^{h_{m }}\psi(t)|\\
&&\qquad \le h!\,\|\tilde v\|_{C^h([\psi(0),\psi(1)])}\sum_{j=0}^h
\sum _{m =1}^{j}
\sum _{{h_{1},\cdots ,h_{m }\ge1}\atop{h_{1}+\cdots +h_{m }=j}}
\|\psi(t)\|_{C^h([0,1])}^m\\&&\qquad =C_{h,\psi}\,\|\tilde v\|_{C^h([\psi(0),\psi(1)])},
\end{eqnarray*}
for a suitable constant~$C_{h,\psi}>0$.

As a result, we have that
$$ \| u-f \|_{C^h([0,1])}\le C_{h,\psi}\,\|\tilde u-\tilde f\|_{C^h([\psi(0),\psi(1)])}\le
C_{h,\psi}\,\epsilon,$$
which is the approximation estimate claimed in Corollary~\ref{densepsi}
(up to renaming~$\epsilon$).

Furthermore, 
for any~$t\ge0$, we have that~$\psi(t)\ge\psi(0)$
and, as a consequence, by Lemma~\ref{capty}
and~\eqref{SUSIS},
$$D_a^{\alpha,\psi}u(t)=D_{\psi(a)}^\alpha (u\circ\psi^{-1})(\psi(t))=
D_{\tilde a}^\alpha \tilde u(\psi(t))=0,$$
thus completing the proof of Corollary~\ref{densepsi}.
\end{proof}

\begin{appendix}

\section{Caputo-stationary functions with vanishing $k$th derivatives
near~$-\infty$}

In this appendix, we remark that Caputo-stationary functions with initial point~$-\infty$
that have vanishing $k$th derivative
near~$-\infty$ are also Caputo-stationary for a fixed point
beyond its constancy interval. Namely, we have that:

\begin{lemma}\label{AfavaadfeAU}
Let~$a\in\mathbb{R}$. Let~$I\Subset(a,+\infty)$ be an interval.
Let~$k\in\mathbb{N}$ and~$\alpha\in(k-1,k)$, and assume 
that~$
u\in C^{k,\alpha}_{-\infty}$,
and that~$u^{(k)}=0$
in~$(-\infty,a)$.

Then, 
\begin{eqnarray}
\label{AIAKA1}&&u\in C^{k,\alpha}_{a}\\
\label{AIAKA2} {\mbox{and }}\quad &&D^\alpha_{a} u=D^\alpha_{-\infty}u\quad{\mbox{ in }}I.
\end{eqnarray}
\end{lemma}

\begin{proof} By~\eqref{CkAM2}, we see that if~$b\in(-\infty,a]\cup\{-\infty\}$,
then~$C^{k,\alpha}_b\subseteq C^{k,\alpha}_a$, and so~\eqref{AIAKA1} plainly follows.
Furthermore, $u^{(k)}$ vanishes in~$(-\infty,a)$, and consequently,
for any~$t\in I$,
$$ 0=
\int_{-\infty}^t \frac{u^{(k)}\left(\tau\right)}{(t-\tau)^{\alpha-k+1}} d\tau
=\int_a^t \frac{u^{(k)}\left(\tau\right)}{(t-\tau)^{\alpha-k+1}} d\tau,$$
which proves~\eqref{AIAKA2}.
\end{proof}

A counterpart of Lemma~\ref{AfavaadfeAU}
allows us to extend a function with its Taylor polynomial maintaining
its Caputo derivative. For this, we first point out that this operation
is compatible with the functional setting
in~\eqref{CkAM2}:

\begin{lemma}\label{LA2lem}
Let~$a\in\mathbb{R}\cup\{-\infty\}$ and~$b\in(a,+\infty)$. 
Let~$k\in\mathbb{N}$ and~$\alpha\in(k-1,k)$.
Let~$f\in C^{k,\alpha}_a$, $g\in C^{k,\alpha}_b$ and assume that
\begin{equation}\label{GLaU}
f^{(j)}(b)=g^{(j)}(b)\qquad{\mbox{for all }}j\in\{ 0,\dots,k-1\}.
\end{equation}
Let
$$ \overline{(a,+\infty)}\ni t\mapsto h(t):=\begin{cases}
f(t) & {\mbox{ if }}t\in\overline{(a,b)},\\
g(t) & {\mbox{ if }}t\in(b,+\infty).
\end{cases}$$
Then~$h\in C^{k,\alpha}_a$.
\end{lemma}

\begin{proof} Since~$f\in C^{k-1}\big( \overline{(a,+\infty)}\big)$
and~$g\in C^{k-1}([b,+\infty))$, we obtain from~\eqref{GLaU}
that~$h\in C^{k-1} \big( \overline{(a,+\infty)}\big)$, and, for every~$t\in\overline{(a,+\infty)}$
and~$j\in\{0,\dots,k-1\}$,
$$ h^{(j)}(t)=\begin{cases}
f^{(j)}(t) & {\mbox{ if }}t\in\overline{(a,b)},\\
g^{(j)}(t) & {\mbox{ if }}t\in(b,+\infty).\end{cases}$$
In particular, we see from~\eqref{GLaU} that
\begin{equation}\label{uNA}
h^{(j)}(b)=f^{(j)}(b)=g^{(j)}(b),\qquad{\mbox{for all }}j\in\{0,\dots,k-1\}.
\end{equation}
Using that~$f^{(j)}\in AC\big( \overline{(a,b)}\big)$
for each~$j\in\{0,\dots,k-1\}$, we can write that,
for every~$t_1,t_2\in\overline{(a,b)}$,
$$ f^{(j)}(t_2)-f^{(j)}(t_1)=\int_{t_1}^{t_2} F_j(\tau)\,d\tau,$$
for a suitable Lebesgue integrable function~$F_j$.

Similarly, if~$T>b$, since~$g^{(j)}\in AC([b,T])$, we have that
for every~$t_1,t_2\in [b,T]$,
$$ g^{(j)}(t_2)-g^{(j)}(t_1)=\int_{t_1}^{t_2} G_j(\tau)\,d\tau,$$
for a suitable Lebesgue integrable function~$G_j$.

Then, given~$T>b$, we define
\begin{equation}\label{hj}
H_j(t)=\begin{cases}
F_j(t) & {\mbox{ if }}t\in\overline{(a,b)},\\
G_j(t) & {\mbox{ if }}t\in(b,T].\end{cases}\end{equation}
We have that~$H_j$ is Lebesgue integrable
and, if~$t_1\in\overline{(a,b)}$ and~$t_2\in(b,T]$, recalling~\eqref{uNA}
we see that
\begin{eqnarray*}
h^{(j)}(t_2)-h^{(j)}(t_1)&=&g^{(j)}(t_2)-f^{(j)}(t_1)\\
&=&g^{(j)}(t_2)-g^{(j)}(b)+f^{(j)}(b)-f^{(j)}(t_1)\\
&=& \int^{t_2}_{b} G_j(\tau)\,d\tau+\int_{t_1}^{b} F_j(\tau)\,d\tau\\
&=& \int_{t_1}^{t_2} H_j(\tau)\,d\tau.
\end{eqnarray*}
{F}rom this, we conclude that
\begin{equation}\label{8iJMAhAAAA}
h^{(j)}\in AC\big( \overline{(a,T)}\big)\qquad{\mbox{for all }}j\in\{0,\dots,k-1\}.\end{equation}
Hence, in view of~\eqref{CkAM2}, to complete the proof of the desired
result it remains to check that~$\Theta_{k,\alpha,h,T}\in L^1\big( (a,T)\big)$,
for every~$T>a$, namely that
\begin{equation}\label{8ujANA}
\int_a^{T} \frac{|h^{(k)}(\tau)|}{(T-\tau)^{\alpha-k+1}}\,d\tau<+\infty.
\end{equation}
We remark that here~$h^{(k)}$ is intended
in the Lebesgue sense, being~$h^{(k-1)}\in AC\big( \overline{(a,T)}\big)$,
due to~\eqref{8iJMAhAAAA}. Hence, in the setting of~\eqref{hj},
we have that~$h^{(k)}=H_{k-1}$ and therefore
\begin{equation}
\int_a^{T} \frac{|h^{(k)}(\tau)|}{(T-\tau)^{\alpha-k+1}}\,d\tau=
\int_a^{T} \frac{|H_{k-1}(\tau)|}{(T-\tau)^{\alpha-k+1}}\,d\tau.
\end{equation}
Consequently, if~$T\le b$ we have that
$$ \int_a^{T} \frac{|h^{(k)}(\tau)|}{(T-\tau)^{\alpha-k+1}}\,d\tau=
\int_a^{T} \frac{|F_{k-1}(\tau)|}{(T-\tau)^{\alpha-k+1}}\,d\tau=
\| \Theta_{k,\alpha,f,T}\|_{L^1(a,T)},
$$
which is finite since~$f\in C^{k,\alpha}_a$.

If instead $T>b$, we have that
\begin{eqnarray*}
\int_a^{T} \frac{|h^{(k)}(\tau)|}{(T-\tau)^{\alpha-k+1}}\,d\tau&=&
\int_a^b \frac{|F_{k-1}(\tau)|}{(T-\tau)^{\alpha-k+1}}\,d\tau+
\int_b^{T} \frac{|G_{k-1}(\tau)|}{(T-\tau)^{\alpha-k+1}}\,d\tau\\&\le&
\int_a^b \frac{|F_{k-1}(\tau)|}{(b-\tau)^{\alpha-k+1}}\,d\tau+
\int_b^{T} \frac{|G_{k-1}(\tau)|}{(T-\tau)^{\alpha-k+1}}\,d\tau
\\&=& \| \Theta_{k,\alpha,f,b}\|_{L^1(a,b)}+
\| \Theta_{k,\alpha,g,T}\|_{L^1(b,T)},
\end{eqnarray*}
which are finite since~$f\in C^{k,\alpha}_a$ and~$g\in C^{k,\alpha}_b$.
This completes the proof of~\eqref{8ujANA} and of the desired result.
\end{proof}

With this, we can obtain a
counterpart of Lemma~\ref{AfavaadfeAU}
(which is not explicitly used here, but that can be useful
for further investigations), as follows:

\begin{lemma}\label{uNA.AOKa}
Let~$a\in\mathbb{R}\cup\{-\infty\}$ and~$b\in(a,+\infty)$.
Let~$I\Subset(b,+\infty)$ be an interval.
Let~$k\in\mathbb{N}$ and~$\alpha\in(k-1,k)$, and assume that~$u\in C^{k,\alpha}_b$.

Let also
$$ u_\star(t):=\begin{cases}
u(t) & {\mbox{ if }}t\in[ b,+\infty),\\
\displaystyle\sum_{j=0}^{k-1} \frac{u^{(j)}(b)}{j!}(t-b)^j& {\mbox{ if }}t\in(-\infty, b).
\end{cases}$$
Then, $u_\star\in C^{k,\alpha}_a$
and~$D^\alpha_a u_\star=D^\alpha_b u$ in~$I$.
\end{lemma}

\begin{proof} We apply Lemma~\ref{LA2lem} with
$$ f(t):=\sum_{j=0}^{k-1} \frac{u^{(j)}(b)}{j!}(t-b)^j,$$
$g(t):=u(t)$, and~$h(t):=u_\star(t)$. Notice that, in this setting,
for each~$j\in\{0,\dots,k-1\}$,
we have that~$f^{(j)}(b)=u^{(j)}(b)=g^{(j)}(b)$, and therefore
condition~\eqref{GLaU} is fulfilled. Hence, the use of Lemma~\ref{LA2lem}
gives that~$u_\star\in C^{k,\alpha}_a$, as desired. In addition,
we have that~$u_\star^{(k)}=0$ in~$(-\infty,b)$ and therefore,
if~$t\in I$,
$$ \int_a^t \frac{u^{(k)}_\star\left(\tau\right)}{(t-\tau)^{\alpha-k+1}} d\tau
=
\int_b^t \frac{u^{(k)}_\star\left(\tau\right)}{(t-\tau)^{\alpha-k+1}} d\tau=
\int_b^t \frac{u^{(k)} \left(\tau\right)}{(t-\tau)^{\alpha-k+1}} d\tau,$$
which says that~$D^\alpha_a u_\star(t)=D^\alpha_b u(t)$.
\end{proof}

\end{appendix}

\end{document}